\documentclass{amsart}
\usepackage{amsmath, amssymb, amsthm, amscd, enumitem}
\usepackage[usenames]{color} % for editing
\usepackage[normalem]{ulem} % for editing
\newtheorem{theorem}{Theorem}[section]
\newtheorem{lemma}[theorem]{Lemma}
\newtheorem{corollary}[theorem]{Corollary}
\newtheorem{remark}[theorem]{Remark}
\newtheorem{proposition}[theorem]{Proposition}

\newcommand{\KK}{\Bbbk}
\newcommand{\QQ}{\mathbb{Q}}
\newcommand{\ZZ}{\mathbb{Z}}
\newcommand{\id}{\textnormal{id}}
\newcommand{\trace}{\textnormal{Tr}}
\newcommand{\irr}{\textnormal{Irr}}

\newcommand{\hmodfin}{H\textbf{\textrm{-}mod}_{\textrm{fin}}}
\newcommand{\EE}{\textbf{E}}

\newcommand{\lcm}{\textnormal{lcm}}
\newcommand{\ord}{\textnormal{ord}}
\newcommand{\co}{\textnormal{co}\,}

\newcommand\C[1]{#1\mbox{-\bf{mod}}_{\operatorname{\mathsf{fin}}}}
\newcommand\ld[1]{\,_{#1}}

\renewcommand\o{\otimes}
\renewcommand\to{\rightarrow}
\newcommand{\End}{\textnormal{End}}
\newcommand{\Hom}{\textnormal{Hom}}
\def\a{{\alpha}}
\def\b{{\beta}}
\def\e{{\epsilon}}
\def\w{{\omega}}
\def\k{{\Bbbk}}

\def\Tr{\trace}
\def\inv{^{-1}}
\def\du{^{\vee}}
\def\bidu{^{\vee\vee}}

\newtheorem{mainthm}{Theorem}

\title{Hopf algebras of dimension $2p^2$}
\author{Michael Hilgemann}
\address{Department of Mathematics, Iowa State University, Ames, IA 50011, USA.}
 \email{hilgem1@iastate.edu}
 \author{Siu-Hung Ng}
\address{Department of Mathematics, Iowa State University, Ames, IA 50011, USA.}
 \email{rng@iastate.edu}
 \thanks{The research was partially supported by NSA grant no. H98230-05-1-0020.}
\date{September 15, 2008}
\begin{document}

\begin{abstract} Let $H$ be a non-semisimple Hopf algebra of dimension $2p^2$ over an algebraically closed field $\KK$ of characteristic zero, where $p$ is an odd prime.  We prove that $H$ or $H^*$ is pointed, which completes the classification for Hopf algebras of these dimensions.
\end{abstract}

\maketitle

\section*{Introduction}

In recent years, there have been some developments in the classification of finite-dimensional Hopf algebras over an algebraically closed field $\KK$ of characteristic zero.  For $p$ a prime, Hopf algebras of dimension $p$ were shown by Zhu in \cite{Zhu94} to be isomorphic to the group algebra $\KK[\ZZ_p]$.  Hopf algebras of dimension $p^2$ were completely classified in \cite{Masuoka96} and \cite{Ng02}.  In the semisimple case, they are isomorphic to either $\KK[\ZZ_{p^2}]$ or $\KK[\ZZ_p \times \ZZ_p]$, and in the non-semisimple case, they are the Taft algebras of dimension $p^2$.

Semisimple Hopf algebras of dimension $pq$, where $p,q$ are distinct  primes, were first classified by
 Masuoka in \cite{Mas95} for the case $p=2$. For $p$ odd, they were studied by Gelaki and Westreich in \cite{GW00}, and it was proved later by Etingof and Gelaki in \cite{EG98} that these Hopf algebras must be trivial, i.e. they are isomorphic to group algebras or the duals of group algebras.

    It is, in general, an open question as to whether every Hopf algebra over $\KK$ of dimension $pq$, where $p, q$ are distinct primes, is semisimple. However, there have been some partial results which suggest an affirmative answer to the question. Natale showed that a quasi-triangular Hopf algebra of such dimension is semisimple in \cite{Na4}. In \cite{Ng05} and \cite{Ng08}, the second author proved that there is no non-semisimple Hopf algebra of dimension $2p$ for $p$ an odd prime, or of dimension $pq$ with $2 < p < q \leq 4p+11$. The latter improves the main results in \cite{Ng04} and \cite{EG03}.

These established classifications have provided a foundation for the investigation of a general Hopf algebra over $\KK$ whose dimension is a product of three primes. The study of Hopf algebras of dimension $p^3$ has been carried out by Garc\'{i}a \cite{GG}.  In this paper we address Hopf algebras of dimension $2p^2$, for $p$ an odd prime.

    Masuoka began the classification of semisimple Hopf algebras of dimension $2p^2$ in \cite{Masuoka2p^2}. In the paper, he constructed a semisimple Hopf algebra $\mathcal{B}_0$ (denoted by $A_{\xi}$ in that paper) of dimension $2p^2$, which he showed is the unique semisimple Hopf algebra $H$ of such dimension with $|G(H)|=p^2$.
 Natale completed the classification in \cite{Natalepq^2} by showing that non-trivial semisimple Hopf algebras of dimension $2p^2$ are isomorphic to $\mathcal{B}_0$ or $\mathcal{B}_0^*$. Since $G(\mathcal{B}_0^*)\cong \ZZ_{2p}$, $\mathcal{B}_0$ is not self-dual.

 The Hopf algebra $\mathcal{B}_0^*$ is a smash product Hopf algebra $\KK[D_{2p}]^* \# \KK[\ZZ_p]$, where $D_{2p}$ is the dihedral group of order $2p$ (cf. \cite{Natalepq^2}). The left $\KK[\ZZ_p]$-action on $\KK[D_{2p}]^*$ is induced by the right action $\lhd$ of $\ZZ_p=\langle g \rangle$ on $D_{2p}=\langle a, b \mid b^p=1, a^2=1, aba=b\inv\rangle$ as automorphisms given by
 $$
 b \lhd g =b, \quad a \lhd g =ba\,.
 $$

For the case that $H$ is a pointed non-semisimple Hopf algebra of dimension $2p^2$, it was shown by Andruskiewitsch and Natale in \cite{AN01} that $H$ is isomorphic to one of the four types of Hopf algebras listed in Lemma A.1 of that paper.  Specifically, let $j \in \{1, 2, 4, \cdots, 2p-2 \}$ and $\mu \in \{ 0, 1\}$ such that $\mu = 0$ whenever $j \neq 1$.  Also let $\tau \in \KK$ be a $2p$-th root of unity such that the order of $\tau$ is a multiple of $p$, with the order being $p$ in the case that $j = 1$.  Define $\mathcal{A}(\tau, j, \mu)$ to be the 
Hopf algebra generated by elements $g$ and $x$ as an algebra subject to the relations $$g^{2p} = 1, \: x^{p} = \mu(1 - g^p), \: gx = \tau xg,$$
 with comultiplication given by $$\Delta(g) = g \otimes g, \: \Delta(x) = x \otimes 1 + g^j \otimes x.$$
 Then the pointed non-semisimple Hopf algebras of dimension $2p^2$ are isomorphic to exactly one of the following $4(p-1)$ Hopf algebras:
\begin{enumerate}
\item[\rm (i)] $\mathcal{A}(\omega, 1, 0)$, $\omega$ a primitive $p$-th root of unity.
\item[\rm (ii)] $\mathcal{A}(\omega, 1, 1)$, $\omega$ a primitive $p$-th root of unity.
\item[\rm (iii)] $\mathcal{A}(\tau, 2r, 0)$,  $1 \le r \le p-1$, where $\tau$ is  a fixed primitive $2p$-th root of unity.
\item[\rm (iv)] $\mathcal{A}(\omega, 2, 0) \cong T_p \otimes \KK[\ZZ_2]$, $\omega$ a primitive $p$-th root of unity.
\end{enumerate}
Here $T_p$ is a Taft algebra of dimension $p^2$. It has been pointed out in \cite{AN01} that the dual of $\mathcal{A}(\omega, 1, 0)$ is isomorphic to a Hopf algebra of type (iii), and the Hopf algebras of type (iv) are self-dual. Therefore, the dual of a Hopf algebra of type (i), (iii), or (iv) is also pointed.  However, the dual of $\mathcal{A}(\omega, 1, 1)$ is not pointed (cf. \cite{Rad75}).

In \cite{Fukuda}, Fukuda has shown that non-semisimple Hopf algebras of dimension 18 are pointed or have a pointed dual.  In this paper, we complete the classification of Hopf algebras of dimension $2p^2$ by proving the following main theorem:

\begin{mainthm}\label{t:main} If $H$ is a non-semisimple Hopf algebra of dimension $2p^2$ over an algebraically closed field of characteristic zero, where $p$ is an odd prime, then $H$ or $H^*$ is pointed.
\end{mainthm}

In particular, this implies that there are exactly $5(p-1)$ isomorphism classes of non-semisimple Hopf algebras of dimension $2p^2$ over $\KK$, and we summarize the classification as follows:

\begin{mainthm} Let $H$ be a $2p^2$-dimensional Hopf algebra over an algebraically closed field $\KK$ of characteristic zero, where $p$ is an odd prime.  Then one of the following holds:
\begin{enumerate}
	\item[\rm (a)] $H$ is trivial.
	\item[\rm (b)] $H$ is non-trivial and semisimple and hence isomorphic to either $\mathcal{B}_0$ or $\mathcal{B}_0^*$.
	 \item[\rm (c)] $H$ is non-semisimple and pointed, and  hence $H$ is isomorphic to one of the $4(p-1)$ Hopf algebras of the form $\mathcal{A}(\omega, 1, 0)$, $\mathcal{A}(\omega, 1, 1)$, $\mathcal{A}(\tau, 2r, 0)$, or $\mathcal{A}(\omega, 2, 0)$, as described above.
   \item[\rm (d)] $H$ is neither semisimple nor pointed. In this case, $H \cong \mathcal{A}(\omega, 1, 1)^*$ for some primitive $p$-th root of unity $\omega \in \KK$.
\end{enumerate}
\end{mainthm}

The remainder of this paper is organized as follows.  We start in Section 1 with some notation and preliminary results that will be useful in proving the main theorem.  We show in Section 2 that a non-semisimple Hopf algebra of dimension $2p^2$ over an algebraically closed field $\KK$ must have an antipode of order $2p$.  Finally in Section 3, we complete the classification of Hopf algebras of dimension $2p^2$ by showing that non-semisimple Hopf algebras of these dimensions are pointed or have a dual which is pointed.

Throughout this paper, $\KK$ denotes an algebraically closed field of characteristic zero unless specifically stated otherwise. The tensor product $\o_\KK$ will be simply denoted by $\o$. The readers are referred to \cite{Montgomery} and \cite{Sweedler} for elementary properties of Hopf algebras.

\section{Notation and Preliminaries}

Let $H$ be a finite-dimensional Hopf algebra over $\KK$ with comultiplication $\Delta$, counit $\epsilon$, and antipode $S$. We will use Sweedler's notation for comultiplication, $\Delta(h) = \sum_{(h)} h_1 \otimes h_2$, or more simply $\Delta(h) = h_1 \otimes h_2$, supressing the summation.  The coassociativity of $\Delta$ induces two natural $H^*$-actions, $\rightharpoonup$ and $\leftharpoonup$, on $H$ given by
$$
f \rightharpoonup h =\sum_{(h)} h_1 f(h_2), \quad \text{and} \quad h \leftharpoonup f = \sum_{(h)} f(h_1) h_2 \quad \text{for } h\in H, \, f \in H^*\,.
$$
A left integral of $H$ is an element $\Lambda \in H$ such that $h\Lambda =\epsilon (h)\Lambda$, and a right integral of $H$ can be defined similarly. The subspace of left (or right) integrals of $H$ is always 1-dimensional. One can easily see that an element $\lambda \in H^*$ is a right
integral if, and only if,
\begin{equation}
  h \leftharpoonup \lambda = \lambda(h)1_H\quad \text{for all }h \in H\,.
\end{equation}
A non-zero element  $a \in H$ is said to be  \emph{group-like} if $\Delta(a) = a \otimes a$. The set of all group-like elements of $H$, denoted by $G(H)$, forms a group under the multiplication of $H$, and  is linearly independent over $\k$. The \emph{distinguished group-like} elements $g \in G(H)$ and $\a \in G(H^*)$ are respectively defined by the conditions:
$$
\Lambda h =\a(h) \Lambda, \quad \lambda \rightharpoonup h= \lambda(h) g\quad \text{for all }h \in H,
$$
where $\lambda \in H^*$ is a non-zero right integral, and $\Lambda \in H$ is a non-zero left integral.
These distinguished group-like elements are closely related to the antipode as indicated by
the celebrated formula of Radford \cite{Radford76}:
\begin{equation}\label{radfordeq}
	S^4(h) = g(\alpha \rightharpoonup h \leftharpoonup \alpha^{-1})g^{-1} \quad \text{for all } h\in H.
\end{equation}
In particular, we find
\begin{equation}\label{eq:ordS^2}
\ord (S^2) \mid 2 \cdot \lcm (\ord(g), \ord(\a))\,.
\end{equation}

The square of the antipode indeed determines the semisimplicity of a Hopf algebra. The following result, which will be used tacitly in the remaining discussion, was mainly proved by Larson and Radford in \cite{LaRa87} and \cite{LaRa88} (see also \cite{Mon01} for an alternative treatment of the result by Montgomery).
\begin{theorem}\label{semisimple} The following statements on a finite-dimensional Hopf algebra $H$ over an algebraically closed field of characteristic zero with antipode $S$ are equivalent.
\begin{enumerate}
	\item[\rm (a)] $H$ is semisimple.
	\item[\rm (b)] $H^*$ is semisimple.
	\item[\rm (c)] $\trace(S^2) \neq 0$.
	\item[\rm (d)] $S^2 = \id_H$. \qed
\end{enumerate}
\end{theorem}

Let $a \in H$ be a group-like element of order $m$. Then $\KK[a]$ is a commutative Hopf subalgebra of $H$. Let $\omega \in \KK$ be a primitive $m$-th root of unity and define
\begin{equation}\label{eq:idempotent}
e_{a,j} = \frac{1}{m}\sum_{i\in \ZZ_m} \omega^{-ij}a^i\quad\text{for $j \in \ZZ_m$.}
\end{equation}
Then $\{e_{a,j}\}_{j \in \ZZ_m}$ is a complete set of primitive idempotents of $\KK[a]$, i.e.
$$
e_{a,i}\cdot e_{a,j} =\delta_{ij} e_{a,i}, \quad\sum_{j\in \ZZ_m} e_{a,j}=1\,.
$$
Therefore, for any right $\KK[a]$-module $V$, we always have the $\KK[a]$-module decomposition:
$$
V= \bigoplus_{i \in \ZZ_m} Ve_{a,i} \,.
$$
Note that $\KK[a]=\bigoplus_{j \in\ZZ_m} \KK e_{a,j}$ as ideals of $\KK[a]$.
If $V$ is a free right $\KK[a]$-module, then $\ord(a) \mid \dim V$ and
$$
\dim V e_{a,j} = \frac{\dim V}{\ord (a)}\,.
$$
Obviously, the same conclusion can be drawn for left $\KK[a]$-modules $V$. Moreover, one can construct the complete set of primitive idempotents $\{e_{\beta,j}\}$ for each group-like element $\beta \in H^*$  in the same way.
\begin{lemma}\label{tracelemma}
  Let $H$ be a finite-dimensional Hopf algebra over $\KK$ with antipode $S$. Suppose $a \in G(H)$ and $\beta \in G(H^*)$ such that $\b(a)=1$, and write $H_{ij}$ for $H e_{a,i}\leftharpoonup e_{\beta,j}$
  for  $i \in \ZZ_{\ord(a)}$ and $j \in \ZZ_{\ord(\beta)}$.
  \begin{enumerate}
    \item[\rm (i)]  We have the $\KK$-linear space decomposition $H =\bigoplus\limits_{i,j} H_{ij}$, and
    $$
  \dim H_{ij} = \frac{\dim H}{\ord(a) \cdot \ord(\beta)} \quad\text{for all } i,j\,.
   $$
    \item[\rm (ii)] If $H$ is not semisimple, then $\Tr(S^2|_{H_{ij}})=0$ for all $i,j$.
  \end{enumerate}
\end{lemma}
\begin{proof}
  (i) By the Nichols-Zoeller Theorem (see \cite{NZ}), $H$ is a free right $\KK[a]$-module. It follows by the preceding remark that $\dim He_{a,i}=\dim H/\ord(a)$ for $i\in \ZZ_{\ord(a)}$, and
  $$
  H=\bigoplus_{i\in \ZZ_{\ord(a)}} H e_{a,i}\,.
  $$
  Since $\beta(a)=1$,
  $$
  (ha^i)\leftharpoonup \beta = (h\leftharpoonup\beta) a^i\quad \text{for all integers $i$ and $h \in H$.}
  $$
  Let $e=e_{a,i}$ for some  $i\in \ZZ_{\ord(a)}$. Then  $(h e) \leftharpoonup \beta =(h\leftharpoonup\beta)e$ for $h \in H$. In particular, $He$ is a right $\KK[\beta]$-module under the action $\leftharpoonup$. Consider the left $\KK[\beta]$-module action on $He$ defined by
  $$
  \beta \rightharpoondown x = x \leftharpoonup (\beta\inv )\quad \text{for all $x \in He$}.
  $$
   Since $He$ is a left $H$-module, it admits a natural right $H^*$-comodule structure $\rho: He \to He \o H^*$ with  $\rho(x) = \sum x^{(0)} \otimes x^{(1)}$ defined by $hx = \sum x^{(0)} x^{(1)}(h)$ for all $h \in H$. It is straightforward to check that $\rho(\beta \rightharpoondown x) = \beta\cdot\rho(x)$ for $x \in He$. Hence,  $He$ is a Hopf module in ${_{\KK[\beta]}\mathcal{M}^{H^*}}$. By the Nichols-Zoeller Theorem, we find $He$ is a free left $\KK[\beta]$-module under the action $\rightharpoondown$. By the preceding remark again, we have
  $$
  He = \bigoplus_j e_{\beta,j}\rightharpoondown He, \quad \text{and}\quad
   \dim \left(e_{\beta,j}\rightharpoondown He\right) =\frac{\dim He}{\ord(\beta)}=\frac{\dim H}{\ord(a)\cdot\ord(\beta)}\,.
  $$
  Since $e_{\beta,j} \rightharpoondown x = x \leftharpoonup e_{\beta, -j}$,  we obtain
  $$
  He = \bigoplus_j H_{ij}, \quad \text{and}\quad
  \dim H_{ij} =\frac{\dim H}{\ord(a)\cdot\ord(\beta)}\,.
  $$
(ii) Let $E_i$ and $F_j$ respectively denote the $\KK$-linear operators on $H$ defined by
$$
E_i(h) = h e_{a,i}, \quad F_j(h) = h \leftharpoonup e_{\beta,j}\quad\text{for $i \in \ZZ_{\ord(a)}$,
$j \in \ZZ_{\ord(\beta)}$.}
$$
Then both $E_i$ and $F_j$ are projections on $H$ and $H_{ij} = F_j E_i(H)$. Since $S^2(a)=a$ and $\beta\circ S^2=\beta$, the operator $S^2$ commutes with both $E_i$ and $F_j$. Therefore, $S^2(H_{ij})\subseteq H_{ij}$, and so 
$$
\Tr(S^2|_{H_{ij}})= \Tr(S^2 \circ F_j \circ E_i) = \Tr(E_i \circ S^2 \circ F_j ).
$$
 It follows by \cite[Proposition 2(a)]{radfordtrace} that
\begin{equation}\label{eq:TrS^2}
\Tr(S^2|_{H_{ij}}) = \lambda(e_{a,i})e_{\beta,j}(\Lambda) = \sum_{l,k} \gamma_{lk} \lambda(a^l) \beta^k(\Lambda)
\end{equation}
for some $\gamma_{lk} \in \KK$, where $\lambda \in H^*$ is a right integral, and $\Lambda \in H$ is a left integral such that $\lambda(\Lambda)=1$. The properties of integrals imply the equalities:
\begin{equation}\label{eq:lL}
\lambda(a^l) a^l = \lambda(a^l) 1, \quad \beta^k(\Lambda) \beta^k = \beta^k(\Lambda)\epsilon
\end{equation}
for all integers $l,k$. Since $H$ is not semisimple, $\lambda(1_H)=\e(\Lambda)=0$.
If $a^l$ is not trivial, then the first equality of \eqref{eq:lL} and the linear independence of distinct group-like elements imply
that $\lambda(a^l)=0$. Similarly, we can also conclude $\beta^k(\Lambda)=0$ for all integers $k$. In view of \eqref{eq:TrS^2}, $\Tr(S^2|_{H_{ij}})=0$ for all $i,j$.
\end{proof}

 Recall that a Hopf subalgebra $A$ of a finite-dimensional Hopf algebra $H$ is \textit{normal} if, and only if, $\sum_{(h)}h_1 a S(h_2), \sum_{(h)}S(h_1) a h_2 \in A$ for all $h \in H$ and $a \in A$.  Since $H$ is finite-dimensional, $A$ is a normal Hopf subalgebra of $H$ if, and only if, $A^+ H = H A^+$, which is in turn equivalent to $A^+ H \subseteq H A^+$ or $H A^+ \subseteq A^+ H$ where $A^+ = A \cap \ker \epsilon$ (see \cite[Section 3.4]{Montgomery} or \cite[Corollary 16]{Ng98} for example).

 Following \cite{Sch93},  we say a sequence of finite-dimensional Hopf algebras
 \begin{equation}\label{eq:exact1}
    1 \to A \xrightarrow{\iota} H \xrightarrow{\pi} B \to 1
 \end{equation}
 is \emph{exact} if $\iota$ and $\pi$ are respectively injective and surjective  Hopf algebra maps such that $\iota(A)$ is normal in $H$ and $\ker \pi =H\iota(A^+)$. In this case, $H$ is called an \emph{extension} of $B$ by $A$. Moreover, the dual sequence
 \begin{equation}\label{eq:exact2}
    1 \to B^* \xrightarrow{\pi^*} H^* \xrightarrow{\iota^*} A^* \to 1
 \end{equation}
 of \eqref{eq:exact1} is also an exact sequence of Hopf algebras.

 The next sequence of propositions on some extensions of finite-dimensional Hopf algebras
 will be used to show the main theorem in the next two sections, and may be of interest in their own right.

\begin{proposition}\label{2dimensional} Let $H, A$ be finite-dimensional Hopf algebras over $\KK$ of dimension $2n$ and $n$, respectively, where $n$ is an odd integer.

\begin{enumerate}[label=(\alph*)]
\item[\rm (a)] If there exists a Hopf algebra surjection $\pi: H \to A$, then $$R = H^{\co\, \pi} = \{ h \in H \mid (\id_H \otimes \pi)\Delta(H) = h \otimes 1_A \}$$ is a normal Hopf subalgebra of $H$ isomorphic to $\KK[\ZZ_2]$.
\item[\rm (b)] If $A$ is a Hopf subalgebra of $H$, then $A$ is a normal Hopf subalgebra of $H$.
\end{enumerate}
\end{proposition}

\begin{proof} (a) It is well-known that $R$ is a left coideal subalgebra of $H$ of dimension 2, and $HR^+ \subseteq R^+H$. In view of the preceding remark, it suffices to show that $R$
contains a non-trivial group-like element. Let $x$ be a non-zero element of $R$ such that $\epsilon(x)=0$. Then $\{1,x\}$ is a basis for $R$ and $\Delta(x) = a \otimes 1 + b \otimes x$ for some $a, b \in H$. By applying $\id_H \otimes \epsilon$ and $\epsilon \otimes \id_H$ to $\Delta(x)$, we find $a = x$ and $\epsilon(b)=1$.  Noting that $(\id_H \otimes \Delta)\Delta(x) = (\Delta \otimes \id_H)\Delta(x)$, we obtain that $b$ is a group-like element of $H$, and so $x$ is a $(1,b)$-skew primitive element. Since $H$ is finite-dimensional and $x \neq 0$, $b \neq 1$. Note that $\{x\}$ is a basis for $\ker\epsilon|_R$, and $bxb\inv \in R$. Since $\epsilon(bxb^{-1}) = 0$, $bxb^{-1} = \zeta x$ for some primitive $M$-th root of unity $\zeta \in \KK$ where $M \mid \ord(b)$.

Suppose $x$ is  a non-trivial $(1, b)$-skew primitive element. Then, by \cite[Proposition 1.8]{AN01}, $\zeta \neq 1$ and $x, b$ generate a Hopf subalgebra $K \cong K_\mu(\ord(b), \zeta)$ for some $\mu \in \{ 0, 1\}$. In particular, $$M^2 \mid \dim K \mid \dim H.$$

On the other hand, since $\epsilon(x^2) = 0$,  $x^2 = \gamma x$ for some $\gamma \in \KK$.  Therefore, $(bxb^{-1})^2 = \gamma bxb^{-1}$ which  implies $\gamma \zeta^2 =\gamma \zeta$. Thus $\gamma = 0$ and hence $x^2=0$. As a consequence, $K \cong K_0(\ord(b), \zeta)$ and $M=2$ is the nilpotency index of $x$. In particular,
  $4 \mid \dim H$, a contradiction.  Therefore, $x$ must be a trivial $(1, b)$-skew primitive element, and hence $x = \nu (1 - b)$ for some non-zero $\nu \in \KK$. This implies $b \in R$, and so $R = \KK[b]$.

(b) Consider the dual of the inclusion map $A \hookrightarrow H$, which is a Hopf algebra surjection $\pi: H^* \to A^*$.  By part (a), $(H^*)^{\text{co}\: \pi} \cong \KK[\ZZ_2]$ is a normal Hopf subalgebra of $H^*$, and so we have the following exact sequence of Hopf algebras: $$ 1 \to \KK[\ZZ_2] \to H^* \stackrel{\pi}{\rightarrow} A^* \to 1.$$  Dualizing this exact sequence, the result follows.
\end{proof}

\begin{remark}
 Susan Montgomery has pointed out that Proposition \ref{2dimensional} (b) is a generalization of \cite[Proposition 2]{KM} to non-semisimple Hopf algebras of dimension $2n$ with $n$ odd.
\end{remark}

The following result generalizes Lemma 5.1 in \cite{Ng02}.  This proposition and its corollary provide a useful way of eliminating certain cases in the next section, especially since we are dealing with non-semisimple Hopf algebras.

\begin{proposition}\label{groupproposition} Let $H$ be a finite-dimensional Hopf algebra over $\KK$.  If there exist semisimple Hopf subalgebras $A \subseteq H$ and $K \subseteq H^*$ such that $\dim H = \left( \dim A \right) \left( \dim K \right)$ and $\gcd\left( \dim A, \dim K \right) = 1$, then $H$ is semisimple.
\end{proposition}

\begin{proof} Define the Hopf algebra surjection $$\pi: H \cong H^{**} \stackrel{i^*}{\rightarrow} K^*$$ where $i: K \to H^*$ is inclusion. As $\gcd( \dim A, \dim K) = 1$, it follows that the image of $A$ under $\pi$ is one-dimensional.  Moreover, $A^+ \subseteq \ker \pi$ and so $$\pi(a) = \pi(a - \epsilon(a)1) + \pi(\epsilon(a)1) = \epsilon(a)\pi(1)$$ for all $a \in A$.  Therefore $A \subseteq H^{\text{co}\: \pi} = \{h \in H \mid h_1 \otimes \pi(h_2) = h \otimes 1_{K^*}\}$.  But $\dim H^{\text{co}\: \pi} = \dim H / \dim K = \dim A$ and so $A = H^{\text{co}\: \pi}$.  It follows by \cite{Schn92} that $H$ is isomorphic a cross product algebra  $A \#_{\sigma} K^*$ for some 2-cocycle $\sigma$.  By \cite[Theorem 2.6]{blattmontg}, the semisimplicity of $A$ and $K^*$ implies the semisimplicity of $H$.
\end{proof}

\begin{corollary}\label{2dimcorollary} Suppose $n$ is an odd integer.  If $H$ is a $2n$-dimensional Hopf algebra over the field $\KK$, and $H$ contains a semisimple Hopf subalgebra of dimension $n$, then $H$ is semisimple.
\end{corollary}

\begin{proof} Let $K$ be a semisimple Hopf subalgebra of $H$ with $\dim K = n$.  Then there is a Hopf algebra surjection $\pi: H^* \to K^*$ and $K^*$ is semisimple.  By Proposition \ref{2dimensional}, we have the exact sequence of Hopf algebras
$$1 \to \KK[\ZZ_2] \to H^* \stackrel{\pi}{\rightarrow} K^* \to 1\,.$$
By Proposition \ref{groupproposition}, $H$ is semisimple.
\end{proof}

The Taft algebras  have been the basic examples of non-semisimple  Hopf algebras (see \cite{Taft71}). The next proposition implies the existence of some semisimple Hopf subalgebra in the dual of an extension of a finite group algebra by a Taft algebra. We will need this result in the proof of Lemma \ref{fourthcase} in the last section.

\begin{proposition}\label{p:es} Let $H$ be a finite-dimensional Hopf algebra over $\KK$ and $A$ a normal Hopf subalgebra of $H$ such that $H/H A^+$ is isomorphic to $\KK[G]$ for some finite group $G$. If the Jacobson radical $J$ of $A$ is a Hopf ideal of $A$, then $HJ$ is a Hopf ideal of $H$, and we have the exact sequence of Hopf algebras:
  $$
  1 \to A/J \to H/H J \to \KK[G] \to 1\,.
  $$
  In particular, $H^*$ admits a semisimple Hopf subalgebra of dimension $|G| \dim (A/J)$.
\end{proposition}
\begin{proof}   Since $H/HA^+ \cong \KK[G]$
  as Hopf algebras, by \cite{Schn92}, the (right) $\KK[G]$-extension $A \subset H$ is $H$-cleft.
  Therefore, there exists a convolution invertible right $\KK[G]$-comodule map $\gamma: \KK[G] \to H$ with the convolution inverse $\overline{\gamma}$ such that $\gamma(1)=1_H$, and $\gamma(g)A\overline\gamma(g) \subseteq A$ for all $g \in G$, and $\sigma(g,h) = \gamma(g)\gamma(h)\overline\gamma(gh)$ for $g,h \in G$ defines a 2-cocycle on $G$ with coefficients in $A$. Moreover, the $\KK$-linear map $\Phi: A \#_\sigma \KK[G] \to H$ defined by $a \#g \mapsto a \gamma(g)$ is an algebra isomorphism (cf. \cite{Montgomery}). In particular, $H = \bigoplus_{g \in G} A \gamma(g)$ as  $\KK$-linear spaces.

   Notice that $\gamma(g)$ is an invertible element of $H$ with inverse $\overline{\gamma}(g)$ for all $g \in G$, since $\gamma \ast \overline{\gamma}(g) = 1_H$. Therefore, $a\mapsto \gamma(g)a \overline{\gamma}(g)$ defines an algebra automorphism on $A$. In particular, we find $J = \gamma(g) J \overline{\gamma}(g)$. Thus,
  $$
  \gamma(g) A = A \gamma(g), \quad \gamma(g) J = J \gamma(g) \quad \text{for all }g \in G\,.
  $$
  Therefore,
  $$
  JH = \sum_{g \in G} JA\gamma(g) = \sum_{g \in G} J\gamma(g) =\sum_{g \in G} \gamma(g)J
  =\sum_{g \in G}  \gamma(g)A J = HJ,
  $$
  and hence $HJ$ is a nilpotent Hopf ideal of $H$.

  If $a \in A \cap HJ$, then $Aa \subseteq HJ$ is also nilpotent, and so $a \in J$. Therefore, the natural map $\iota: A/J \to H/HJ$ induced from the inclusion map $i: A \to H$ is also injective. Since $A$ is normal in $H$, $\iota(A/J)$ is normal in $H/HJ$. Let $\pi: H/HJ \to H/HA^+$ be the natural surjection. It is immediate to  see that $\ker \pi = HA^+/HJ = (H/HJ)\iota(A/J)^+$. Therefore, the sequence of finite-dimensional Hopf algebras
  $$
  1 \to A/J \xrightarrow{\iota} H/H J \xrightarrow{\pi} H/HA^+ \to 1
  $$
  is exact. Consequently, $H/H J$ is a crossed product $A/J \#_{\tau} \KK[G]$, and hence semisimple by \cite[Theorem 2.6]{blattmontg}. Moreover, $\dim H/HJ = |G|\dim (A/J)$.  Since  $(H/H J)^*$ is isomorphic to a Hopf subalgebra of $H^*$, the second statement follows.
\end{proof}

We close this section with a few results on linear algebra which will be used frequently together with Lemma \ref{tracelemma} to determine the order of an antipode in the next section.  The first two results are known, and we prove the third.

\begin{lemma}\label{linearlemma} Let $V$ be a finite-dimensional vector space over the field $\KK$, $p$ a prime, and $T$ a linear automorphism on $V$ such that $\trace(T) = 0$.
\begin{enumerate}[label=(\alph*)]
\item[\rm (a)] \cite[Lemma 2.6]{AS98} If $T^{2p} = \id_V$, then $\trace(T^p) = pd$ for some integer $d$.
\item[\rm (b)] \cite[Lemma 1.4]{Ng05} If $T^{p^n}=\id_V$ for some positive integer $n$, then $p$ divides the dimension of $V$.
\item[\rm (c)] If $\dim V=p$ and $T^m = \id_V$ for some positive integer $m = 2^n p$, where $p$ is odd, then $T^p = \xi \: \id_V$ for some $2^n$-th root of unity $\xi \in \KK$.
\end{enumerate}
\end{lemma}

\begin{proof} We prove part (c).  The statement is obviously true for $n=0$. We assume $n \geq 1$. Let $\omega \in \KK$ be a primitive $m$-th root of unity, and
  $V_b$ the eigenspace of $T$ associated to the eigenvalue $\omega^b$. We consider the polynomial $f(x)=\sum_{b=0}^{m-1} (\dim V_b) x^b \in \ZZ[x]$. Since
\begin{equation*}\label{traceeq}
 0 = \trace(T) = \sum_{b=0}^{m-1}(\dim V_b) \omega^b = f(\omega),
\end{equation*}
we have $f(x) = g(x) \Phi_m(x)$ for some $g(x) \in \ZZ[x]$ where $\Phi_k$ denotes the $k$-th cyclotomic
polynomial. Therefore,
\begin{equation}\label{eq:div by p}
\trace(T^p) = f(\omega^p) = g(\omega^p)\Phi_m(\omega^p)\,.
\end{equation}
Note that $\{\omega^{p i} \mid i=0, \dots, 2^{n-1}-1\}$ is a basis for $\QQ(\omega^p)$ and
$\Phi_m(x) = \Phi_p(-x^{2^{n-1}})$. Therefore,
\begin{equation}\label{cyclotomic}
\Phi_m(\omega^p) = \Phi_p(-(\omega^p)^{2^{n-1}}) =\Phi_p(1) = p.
\end{equation}
Let $W_i$ be the eigenspace of $T^p$ associated to the eigenvalue $\omega^{pi}$. Since
$$
\omega^{p(i+2^{n-1})} =- \omega^{pi} \: \text{for } i=0, \cdots, 2^{n-1}-1,
$$
we have
$$
\trace(T^p) = \sum_{i=0}^{2^{n-1}-1} (\dim  W_i - \dim W_{2^{n-1}+i})\omega^{pi}\,.
$$
There exists $i$ such that $\dim  W_i - \dim W_{2^{n-1}+i} \neq 0$ otherwise $\dim V$ is even.
By equations \eqref{eq:div by p} and \eqref{cyclotomic}, $p \mid \dim  W_i - \dim W_{2^{n-1}+i}$. Since
$\dim V = p$, only one of the eigenspaces $W_i$, $W_{2^{n-1}+i}$ is non-zero and any other eigenspace of $T^p$ is trivial. Thus $T^p = \xi \: \id_V$ for some $2^n$-th root of unity $\xi \in \KK$.
\end{proof}

\section{The order of the antipode}
{\bf Throughout the remainder of this paper, we will assume that $H$ is a non-semisimple Hopf algebra of dimension $2p^2$ over  $\KK$ with antipode $S$, where $p$ is an odd prime}. We will prove in this section that the antipodes of these Hopf algebras are of order $2p$ (Theorem \ref{t:3}).

 To establish this result, we first consider the distinguished group-like elements $g \in H$ and $\a \in H^*$. Since $4 \nmid \dim H = 2p^2$, it follows by \cite[Corollary 2.2]{Ng05} that  one of the distinguished group-like elements $g$ or $\a$ is non-trivial. In view of the Nichols-Zoeller Theorem,
 we have
 $$
 \lcm( \ord(g), \ord(\alpha) ) = 2, p, 2p, p^2 \text{ or } 2p^2\,.
 $$

  Without loss of generality, we may assume
 \begin{equation}\label{eq:assumption}
   \ord(g)\ge \ord(\a)
 \end{equation}
 by duality. Under this assumption, $\ord(g) > 1$.

 Let us write $e_i$ for the idempotent $e_{g,i} \in \KK[g]$ defined in \eqref{eq:idempotent}, and $f_j$ for $e_{\a, j} \in \KK[\a]$. We define
 \begin{equation}\label{eq:Hij}
    H_{ij} = He_i \leftharpoonup f_j \quad \text{for all }i \in \ZZ_{\ord(g)} \text{ and  } j \in \ZZ_{\ord(\a)}\,.
 \end{equation}
 It follows by Lemma \ref{tracelemma} that
 $$
 \dim He_i = \frac{\dim H}{\ord(g)}\,, \quad \text{and}\quad  \Tr(S^2|_{He_i})=0
 $$
 for all $i \in \ZZ_{\ord(g)}$. If $\a(g)=1$, then we also have
 $$
 \Tr(S^2|_{H_{ij}})=0 \quad \text{for all }i,j.
 $$

 We first eliminate those values of  $\lcm( \ord(g), \ord(\alpha) )$ which are not possible.
\begin{lemma}
The only possible values of $\lcm(\ord(g),\ord(\alpha))$ are $p$ and $2p$.
\end{lemma}

\begin{proof} Assume that either $\lcm(\ord(g), \ord(\alpha)) = 2p^2$ or $p^2$.  Since $\ord(g) \geq \ord(\alpha)$, it follows that $\ord(g) = p^2$ in both cases.  So $g$ generates a $p^2$-dimensional semisimple Hopf subalgebra of $H$.  By Corollary \ref{2dimcorollary}, $H$ must be semisimple, which is a contradiction.

If $\lcm( \ord(g), \ord(\alpha) ) = 2$ then $\ord(g) = 2$ and $\ord(\alpha) = 1$ or $2$. By Lemma \ref{tracelemma}, $\dim H e_i=p^2$ and $\Tr(S^2|_{H e_i})=0$ for $i \in \ZZ_2$.
 Using Radford's formula \eqref{radfordeq} we have that $S^8 = \id_H$.  In particular, $(S^2|_{He_i})^4=\id_{H e_i}$. It follows by Lemma \ref{linearlemma} (b) that $2 \mid \dim H e_i = p^2$, another contradiction.
\end{proof}

Note that if $\lcm(\ord(g),\ord(\alpha)) = p$ or $2p$, then $\ord(g) = p$ or $2p$ as we are assuming $\ord(g) \geq \ord(\alpha)$.  We first eliminate the possibility that $\ord(g) = p$ and $\ord(\alpha) = 2$.

\begin{lemma}
  The pair $(\ord(g), \ord(\a))$ cannot be $(p,2)$.
\end{lemma}
\begin{proof}
  Suppose $\ord(g)=p$ and $\ord(\a)=2$. Then by Radford's formula \eqref{radfordeq}, $S^{8p}=\id_H$. Since $p$ is an odd prime, we find $\a(g)=1$. From Lemma \ref{tracelemma} we get
 $$
 \dim H_{ij} = \frac{\dim H}{\ord(g)\cdot \ord(\a)} = p, \quad \Tr(S^2|_{H_{ij}})=0
 $$
 for all $i \in \ZZ_p$ and $j \in \ZZ_2$, and $H$ is a direct sum of the subspaces $H_{ij}$ defined in \eqref{eq:Hij}. By Lemma \ref{linearlemma}(c) it follows that $(S^2|_{H_{ij}})^p = \zeta_{ij} \id_{H_{ij}}$ for some $4$-th root of unity $\zeta_{ij} \in \KK$. Since $e_i \leftharpoonup f_0=e_i$, and $S^2(e_i)=e_i$, $\zeta_{i0}=1$ for all $i\in \ZZ_p$. Let $V_j = \bigoplus_i He_i \leftharpoonup f_j = H \leftharpoonup f_j$ for $j\in \ZZ_2$. Then $S^{2p}|_{V_0} = \id_{V_0}$. Thus, for $x \in V_0$ we have
 $$
 x = S^{4p}(x) = \a \rightharpoonup x \leftharpoonup\a = \a  \rightharpoonup x\,.
 $$
 Let $L(\a)$ denote the linear operator on $H$ defined by $L(\a)(x) =\a  \rightharpoonup x$. Then $L(\a)^2=\id_H$, $\Tr(L(\a))=0$ and $\Tr(L(\a)|_{V_{0}})=p^2$. Since $L(\a)(V_{1}) \subseteq V_{1}$, $L(\a)|_{V_{1}}=-\id_{V_{1}}$. Therefore, $\a \rightharpoonup x \leftharpoonup\a =x$ for all $x \in V_1$, and hence
 $S^{4p}=\id_H$. In particular,  $H$ is of index $p$. It follows by \cite[Corollary 3.2]{Ng04} that the subspace
 $$
 H_- =\{x \in H \mid S^{2p}(x)=-x\}
 $$
 is of even dimension. Since $S^{2p}|_{V_0}=\id_{V_0}$, $H_- \subseteq V_1$.

 Let $V_{j+}, V_{j-}$ denote  the eigenspaces of $S^2|_{V_j}$ associated to the eigenvalues $1, -1$ respectively. Since the $p$ distinct eigenvalues of $S^2|_{H_{i1}}$ are $\zeta_{i1}, \zeta_{i1}\w, \dots, \zeta_{i1}\w^{p-1}$, where $\w \in \KK$ is a primitive $p$-th root of unity, we find
 $$
 \dim \{x \in V_1\mid S^{2p}(x)=x\}=p \dim V_{1+}, \quad \dim H_- = p\dim V_{1-}\,.
 $$

 Let $\lambda \in H^*$ be a non-zero right integral. We claim that $(x,y)=\lambda(xy)$ defines a non-degenerate alternating form on $V_{1+}$. It follows by \cite{radfordtrace} that
 $$
 (x,y)=\lambda(xy) = \lambda(S^2 (y \leftharpoonup \a)x) = -\lambda(yx)=-(y,x)
 $$
 for all $x,y \in V_{1+}$. Moreover,
 $$
 \lambda(S^2(v))=\a(g)\lambda(v) = \lambda(v)= \lambda(\a \rightharpoonup v) \quad \text{for all $v \in H$}.
  $$
  Therefore, $\lambda(v)=0$ for any $v \in H$ which satisfies $S^2(v)=\mu v$ or $\a \rightharpoonup v=\mu v$ for some  $\mu \in \KK$ not equal to 1. Let $x \in V_{1+}$ such that $\lambda(xy)=0$ for all $y \in V_{1+}$. Then $\lambda(xy')=0$ for all $y' \in H$. By the non-degeneracy of $\lambda$ on $H$, $x=0$ and hence $(\cdot, \cdot)$ is non-degenerate on $V_{1+}$. Thus, $\dim V_{1+}$ is even.
  Note that $p\dim V_{1+}+\dim H_-=\dim V_1=p^2$.  This implies $p^2$
  is even, a contradiction.
\end{proof}

We now proceed to show that $S^{2p} = \id_H$ for the remaining possibilities of $\lcm(\ord(g),\ord(\alpha))$. We will use this to deduce in the next section that $H$ or $H^*$ is pointed.

\begin{lemma}\label{2pcase} If $\lcm(\ord(g), \ord(\a))=2p$, then $S^{2p} = \id_H$.
\end{lemma}
\begin{proof} Under this hypothesis,  $S^{8p}=\id_H$ by Radford's formula \eqref{radfordeq}. Moreover,
either  $\ord(g)=2p$ or  $(\ord(g), \ord(\a))=(p,2)$. The second case can be eliminated by the preceding lemma. Therefore, $\ord(g) = 2p$, and $\dim H e_i = p$ for $i \in \ZZ_{2p}$.  Fix $i \in \ZZ_{2p}$ and set $T_i = S^2 |_{H e_i}$.  Then $T_i^{4p}=\id_{He_i}$ and $\Tr(T_i)=0$. By Lemma \ref{linearlemma}(c), we have $T^p_i = \zeta_i \id_{He_i}$ where $\zeta_i \in \KK$ is a 4-th root of unity. Note that
$$
T_i(e_i) = S^2(e_i) = e_i\,.
$$
Therefore, $\zeta_i =1$ and so $T^p_i=\id_{He_i}$ for all $i \in \ZZ_{2p}$. Hence, we get that
$S^{2p} = \id_H$ as desired.
\end{proof}

We now deal with the remaining case in the following lemma.
\begin{lemma}
  If $\lcm(\ord(g), \ord(\a))=p$, then $S^{2p}=\id_H$.
\end{lemma}
\begin{proof} The proof will be presented in steps (i)-(v).\smallskip\\
(i) Since $\lcm(\ord(g), \ord(\a))=p$,  $S^{4p} =\id_H$ by Radford's formula \eqref{radfordeq}. In particular, the possible eigenvalues of $S^{2p}$ are $\pm 1$. Define
$$
H_{+} = \{ x \in H \mid S^{2p}(x) = x \}, \quad H_{-} = \{ x \in H \mid S^{2p}(x) = -x\}\,.
$$
Then we have
\begin{equation}\label{eq:H+-}
\dim H_+ +\dim H_- = 2p^2\quad \text{and}\quad \Tr(S^{2p}) = \dim H_+ -\dim H_-\,.
\end{equation}
(ii) The pair $(\ord(g), \ord(\a))$ can either be $(p,1)$ or $(p,p)$.  In both cases, $\ord(g)=p$. Hence
$\lcm(\ord(g), \ord(S^4))=p$, or equivalently, $H$ is of index $p$. It follows by \cite[Proposition 1.3]{Ng04} that
$\dim H_-$ is an even integer.\smallskip\\
 (iii) \emph{If $(\ord(g), \ord(\a))=(p,p)$, then  $\a(g)\ne 1$.} Otherwise, we can apply Lemma \ref{tracelemma} and get
 $$
\dim H_{ij} = 2, \quad \Tr(S^2|_{H_{ij}})=0 \quad \text{and}\quad
H=\bigoplus_{i,j \in \ZZ_p} H_{ij}\,,
$$
where $H_{ij}$ are defined in \eqref{eq:Hij}. By (i), we find $\left(S^2|_{H_{ij}}\right)^{2p}=\id_{H_{ij}}$, and so $S^2|_{H_{ij}}$ has exactly two distinct eigenvalues $\pm \xi$ for some $p$-th root of unity $\xi \in \KK$. Therefore, $\pm 1$ are the two distinct eigenvalues of $S^{2p}|_{H_{ij}}$. This implies
$\dim H_- =\dim H_+=p^2$ which contradicts (ii). \smallskip\\
(iv) \emph{We claim that $\Tr(S^{2p}) = p^2d$ for some integer $d$.}
 If $(\ord(g), \ord(\a))=(p,1)$, then  $H$ is a unimodular Hopf algebra of index $p$ by (ii). It follows immediately by \cite[Lemma 4.3]{Ng02} that $\trace\left( S^{2p} \right) = p^2 d$ for some integer $d$.
We may now assume $(\ord(g), \ord(\a))=(p,p)$.
Let $B = \KK[g]$ and $I = \KK[\alpha]^{\perp}$, which is a Hopf ideal of $H$, and  $\overline H = H/I \cong \KK[\alpha]^* \cong B$.  Let $\pi: H \to B$ be the natural surjection of Hopf algebras. By (iii), $\a(g) \ne 1$ and so $g-1 \not\in \KK[\a]^\perp$, or equivalently $\pi(g) \ne 1_{\overline H}$. Since $\dim \overline H =p$,  $\pi(B) = \overline  H$. Therefore, the composition $\pi \circ i : B \to \overline H$ is an isomorphism of Hopf algebra, where $i: B \to H$ is inclusion. Thus, there exists a surjective Hopf algebra map $\nu: H \to B$ such that $\nu \circ i =\id_B$. Hence, by  \cite{radfordproj}, $H$ is isomorphic to the biproduct $R \times B$ where $R$ is the right coinvariant given by
$$
R = H^{\text{co} \: \nu} = \{h \in H \mid (\id_H \otimes \nu)\Delta(h) = h \otimes 1_B\}\,.
$$
Due to the results in \cite[Section 4]{AS98}, $R$ is invariant under $S^2$ and $S^2 = T \otimes S^2 |_B$
where $T = S^2 |_R$ if one identifies $H$ with $R \times  B$.  Since $S^2|_B=\id_B$ and
$$ 0 = \trace(S^2) = \trace(T) \trace(\id_B) = \trace(T)p\,,$$
it follows that
$\trace(T) = 0$.  Also we have that $T^{2p} = \id_R$ as $S^{4p} = \id_H$.  Thus by Lemma \ref{linearlemma}(b), $\trace(T^p) = pd$ for some integer $d$.  Therefore, $$
\trace(S^{2p}) = \trace(T^p)\trace(\id_B) = p^2 d, \quad\text{as claimed.}
$$
(v) The equalities in \eqref{eq:H+-} imply that $d$ is an even integer.  Since $$-2p^2 \leq \trace\left( S^{2p} \right) \leq 2p^2,$$ it follows that $d$ can only be $-2, 0,$ or $2$. Note that if $d = -2$ then $S^{2p} = -\id_H$, which is not possible as $S^{2p}(1_H) = 1_H$.  If $d = 0$ then $\trace\left( S^{2p} \right) = 0$ and hence $\dim H_+ = \dim H_- = p^2$.  But this contradicts  (ii) which asserts that $\dim H_{-}$ is even.  Hence $d = 2$ and so $\trace\left( S^{2p} \right) = 2p^2$ which implies that $S^{2p} = \id_H$, as desired.
\end{proof}

With the beginning remarks and these four lemmas, we have proved the following theorem.
\begin{theorem}\label{t:3}
  Let $p$ be an odd prime, and $H$ a non-semisimple Hopf algebra over $\KK$ of dimension $2p^2$. Then the order of the antipode of $H$ is $2p$.
\end{theorem}

\section{Non-semisimple Hopf algebras of dimension $2p^2$}
In this section, we will show that if $H$ is a non-semisimple Hopf algebra of dimension $2p^2$ over $\KK$, for $p$ an odd prime, then $H$ or $H^*$ is pointed.  This completes the proof of our main result Theorem \ref{t:main}.  Recall that a Hopf algebra $H$ over $\KK$ is \emph{pointed} if its simple subcoalgebras are all 1-dimensional. Therefore, $H^*$ is pointed if all simple $H$-modules are 1-dimensional.

We continue to assume that $H$ is a non-semisimple Hopf algebra of dimension $2p^2$ over the field $\KK$ with the antipode $S$, where $p$ is an odd prime. Again, we let $g \in G(H)$ and $\a \in G(H^*)$ denote the distinguished group-like elements.  From Section 2, we know that the order of the antipode $S$ of $H$ is $2p$, and that
$$
\lcm(\ord(g), \ord(\a))=p \text{ or } 2p\,.
$$
\textbf{By duality, we will assume $\ord(\a) \ge \ord(g)$ in this section}. Therefore, $\ord(\a)=p$ or $2p$. Note that the opposite assumption was made in Section 2.

For any finite-dimensional $H$-module $V$, the left dual $V\du$ of $V$ is the $H$-module with the underlying space $V^*=\Hom_\KK(V, \KK)$ and the $H$-action given by
$$
(h\cdot f)(v) = f(S(h)v)\quad\text{for all }f \in V^* \text{ and } v \in V.
$$
For an algebra automorphism $\sigma$ on $H$, let $\ld{\sigma}V$ be the $H$-module with the underlying space $V$ and the action given by $h \cdot_{\sigma} v = \sigma(h)v$ for all $h \in H$ and $v \in V$.  It is easy to verify that the natural isomorphism $j: V \to V^{**}$ of vector spaces is also an $H$-module map from $\ld{S^2}V$ to $V\bidu$. Hence,
\begin{equation}
  \ld {S^2} V \stackrel{j}{\cong} V\bidu\quad \text{for } V \in \C{H}\,.
\end{equation}

Let  $P(V)$ denote the projective cover of $V \in \C{H}$. For $\beta \in G(H^*)$, we define $\KK_{\beta}$ as the 1-dimensional $H$-module which affords the irreducible character $\beta$. We will simply denote by $\KK$ the trivial 1-dimensional $H$-module $\KK_\epsilon$.

Let us denote by $[V]$  the isomorphism class of a simple $H$-module $V$. The cyclic group $G=\langle \a \rangle$ acts on  the set $\irr(H)$ of all isomorphism classes of simple $H$-modules  by setting $$\beta [V] = [\KK_{\beta} \otimes V]$$ for all $\beta \in G$ and $[V] \in \irr(H)$.  We denote the $G$-orbit of $[V]$ in $\irr(H)$ by $O(V)$. Since $\KK_\beta \o -$ is a $\KK$-linear equivalence on $\C{H}$, we have  $P(\KK_{\beta} \otimes V) \cong \KK_{\beta} \otimes P(V)$ and thus $\dim P(W) = \dim P(V)$ for all $[W] \in O(V)$.  In particular, $\dim P(\KK) = \dim P(\KK_{\beta})$ for all $\beta \in G$. Suppose $\{[V_0], [V_1], \cdots, [V_{t}]\}$ is a complete set of representatives of $G$-orbits in $\irr(H)$ with $V_0 = \KK$. The left $H$-module $H$ has the decomposition of principal modules:
\begin{equation} \label{eq:decomp}
\begin{aligned}
H \cong  \bigoplus_{[V]\in \irr(H)} & (\dim V)\cdot P(V) \\
\cong & \bigoplus_{\b \in \langle \a \rangle} P(\KK_\b) \oplus \bigoplus_{i=1}^t \dim V_i \left(\bigoplus_{[W] \in O(V_i)} P(W)\right)\,.
\end{aligned}
\end{equation}
This decomposition implies the equation:
\begin{equation}\label{dimcount}
	 \dim H  = \ord(\a)\dim P(\KK)+\sum_{i >0 } |O(V_i)|\dim V_i \cdot \dim P(V_i)\,.
\end{equation}
 We can now demonstrate that $\dim P(\KK)$ can only take  two possible values.
\begin{lemma}\label{l:dimP}
The value of $\dim P(\KK)$ can either be  $p$ or $2p$.
\end{lemma}

\begin{proof} We have that $$P(\KK) \cong P(\KK^{\vee \vee}) \cong P(\KK)^{\vee \vee} \cong \:_{S^2}P(\KK)\,.$$ Since $S^{2p} = \id_H$, by Lemma \ref{l:A2} in the Appendix, there exists an $H$-module isomorphism $$\phi : P(\KK) \to \:_{S^2}P(\KK)$$ such that $\phi^p = \id_{P(\KK)}$.  It follows by  \cite[Lemma 1.3]{Ng08} that $\trace(\phi) = 0$. In view of \cite[Lemma 2.6]{AS98} or Lemma \ref{linearlemma}(b), $\dim P(\KK) = np$ for some positive integer $n$.  But then by equation (\ref{dimcount}), $$2p^2 = \dim H \geq p \dim P(\KK) = np^2$$ and so $n = 1$ or $2$.
\end{proof}

In view of Lemma \ref{l:dimP} and the beginning remark of this section, we find
\begin{equation}\label{eq:4cases}
(\ord(\a), \dim P(\KK))=(2p, 2p),\, (p,2p),\, (2p, p)\text{ or } (p,p)\,.
\end{equation}
The following lemma settles the first three cases.

\begin{lemma}\label{secondcase} The pair $(\ord(\a), \dim P(\KK))\ne (2p, 2p)$.
If $(\ord(\a), \dim P(\KK))$ is equal to either $(p, 2p)$ or $(2p,p)$, then $H^*$ is pointed.
\end{lemma}

\begin{proof} Let
$$
H_\a=\bigoplus_{\b \in \langle\a\rangle} P(\KK_{\b})\,.
$$
Then $H_\a$ is isomorphic to a left submodule of ${_HH}$, and so $\dim H_\a \le \dim H = 2p^2$. By the preceding remark,
$\dim H_\a = \ord(\a) \dim P(\KK)$. Therefore, it cannot be possible to have $(\ord(\a), \dim P(\KK))=(2p, 2p)$.
If  $(\ord(\a), \dim P(\KK))=(p, 2p)$ or $(2p,p)$, then $\dim H_\a=\dim H$ and hence $H \cong H_\a$ as left $H$-modules. Since all the simple quotients of $H_\a$ are 1-dimensional,
every simple $H$-module is 1-dimensional. Therefore, $H^*$ is pointed.
\end{proof}

Next we handle the remaining case that $\ord(\a)= \dim P(\KK) = p$. Following the terminology in \cite{EG03}, an $H$-module $V$ is called \textit{$\a$-stable} if $\KK_{\a} \otimes V \cong V$.  We have two subcases; either there exists an $\alpha$-stable simple $H$-module, or all simple $H$-modules are not $\a$-stable.

\begin{lemma}\label{thirdcase} If $\ord(\a) = \dim P(\KK) = p$, and all the simple $H$-modules  are $\a$-stable, then $H^*$ is pointed.
\end{lemma}

\begin{proof} Since $\dim P(\KK) = p=\ord(\a)$, it follows by \eqref{dimcount} that there exists a simple $H$-module $V$ such that $[V] \not\in O(\KK)$. Since
$$
\Hom_H(V\du \o P(V) , \KK) \cong \Hom_H(P(V), V \o \KK)  \cong \Hom_H(P(V), V)
$$
and $\dim \Hom_H(P(V), V)=1$, we get $\dim \Hom_H(V\du \o P(V) , \KK)=1$. Thus, $P(\KK)$ is a direct summand of $V\du \o P(V)$. In particular,
$$
\dim V \dim P(V) = \dim (V\du \o P(V)) \ge \dim P(\KK)=p\,.
$$
Since $V$ is not $\a$-stable, $|O(V)|=p$ and so \eqref{dimcount} implies
$$
2p^2 =\dim H \ge  p\dim P(\KK) + p \dim V \dim P(V) \ge 2p^2\,.
$$
Therefore, $\dim V \dim P(V) = p$ which forces  $\dim V = 1$ and $\dim P(V) = p$.  Thus all simple $H$-modules are 1-dimensional and so $H^*$ is pointed.
\end{proof}

\begin{lemma}\label{fourthcase} If $\ord(g) =\dim P(\KK) = p$, and there exists an $\alpha$-stable simple $H$-module, then $H$ is pointed.
\end{lemma}

\begin{proof} Assume $V$ is an $\alpha$-stable simple $H$-module. Then $O(V)=\{[V]\}$ and $P(V)$ is also $\alpha$-stable. By \cite[Lemma 1.4]{Ng08}, both $\dim V$ and $\dim P(V)$ are multiples of $p$. Let $\dim V = np$ and $\dim P(V) = mp$ for some positive integers $n \le m$. By \eqref{dimcount}, we have the inequality
$$
2p^2 = \dim H \geq p \dim P(\KK) + \dim V \dim P(V) \geq p^2 + nmp^2 = (nm+1)p^2
$$
which implies $n=m=1$. Hence $P(V)= V$ and $\irr(H)=O(\KK)\cup \{[V]\}$ by \eqref{dimcount}. Recall from \cite[Lemma 1.1]{Ng08} that the socle of $P(\KK)$ is $\KK_{\a\inv}$. Therefore,
$$
\Hom_H(P(V), P(\KK)) = \Hom_H(V, P(\KK))= 0\,.
$$
Since $V$ is the only simple $H$-module of dimension greater than 1, all the composition factors
of $P(\KK)$ are 1-dimensional. Now let $\EE$ be the full subcategory of all $M \in \hmodfin$ whose composition factors are 1-dimensional.  So $\EE$ is a proper tensor subcategory of $\hmodfin$ with $P(\KK) \in \EE$, and all the simple objects of $\EE$ are 1-dimensional.  There is a Hopf ideal $I$ of $H$ such that $\EE$ is equivalent to $H/I\textbf{\textrm{-}mod}_{\textrm{fin}}$.   Since $P(\KK)$ is indecomposable of dimension $p$, $H/I$ is not semisimple. Thus, the Hopf algebra $H/I$ is a proper quotient, and it must have dimension $p^2$ as all other possibilities (being $2, p, 2p$) are semisimple Hopf algebras (cf. \cite{Zhu94} and \cite{Ng05}).
It follows by \cite{Ng02} that  $H/I \cong T_p$ where $T_p$ is a Taft algebra of dimension $p^2$.  By Proposition \ref{2dimensional}, we have an exact sequence of Hopf algebras: $$1 \to \KK[\ZZ_2] \to H \to T_p \to 1\,.$$
Dualizing this sequence, we get the exact sequence:
$$1 \to T_p \to H^* \to \KK[\ZZ_2] \to 1$$
since both $T_p$ and $\KK[\ZZ_2]$ are self-dual. It is well-known that the Jacobson radical $J$ of the Taft algebra $T_p$ is a Hopf ideal, and $T_p/J \cong \KK[\ZZ_p]$. Applying Proposition \ref{p:es}, we find $H$ contains a semisimple Hopf subalgebra $K$ of dimension $2p$.  Since $S^{2p}=\id_H$, \cite[Proposition 5.1]{AS98} implies that $K$ is the coradical of $H$.  It follows immediately  by
\cite[Lemma A.2]{AN01} that $H$ is pointed.
\end{proof}

Combining Lemmas \ref{l:dimP}, \ref{secondcase}, \ref{thirdcase}, and \ref{fourthcase}, we complete the proof of our main result Theorem \ref{t:main}.
\begin{center}
{\bf Acknowledgement}
\end{center}
The authors would like to thank Akira Masuoka for his invaluable suggestion on the proof of Proposition \ref{p:es}, and Susan Montgomery for her suggestion of several references. The second author  also  thanks  the National Center for Theoretical Sciences in Taiwan, where the present work was carried out, for the generous hospitality, and Ching-Hung Lam for being the wonderful host.

\section*{Appendix}
\setcounter{section}{1}
\setcounter{theorem}{0}
\def\thesection{\Alph{section}}

The following lemma has been used in \cite[Lemmas 1.4 and 2.2]{Ng08}. We believe the result is known but we provide a proof for the sake of completeness.

\begin{lemma}\label{l:A1}
Let $A$ be an algebra over an algebraically closed field $\KK$, and $\sigma$ a diagonalizable (or semisimple) algebra automorphism on $A$. If $V$ is a finite-dimensional $A$-module isomorphic to  $\ld \sigma V$, then there exists an isomorphism of $A$-modules $\phi: V \to \ld \sigma V$ such that $\phi$ is diagonalizable.
\end{lemma}

\begin{proof}
  Let $\psi: V \to \ld \sigma V$ be an isomorphism of $A$-modules, and $\rho: A \to \End_\KK(V)$ the representation of $A$ associated with the $A$-module $V$. Then we have
  $$
   \psi  \circ \rho(a) = \rho(\sigma(a))\circ \psi\quad \text{for all }a \in A.
  $$
  Suppose $a \in A$ is an eigenvector of $\sigma$ corresponding to the eigenvalue $\gamma$ such that $\rho(a)\ne 0$. Then we have
  $
  C_\psi (\rho(a)) = \gamma \rho(a)
  $
  where $C_\psi: \End_\KK(V) \to \End_\KK(V)$, $g \mapsto \psi \circ g \circ \psi \inv$. In particular, $\rho(a)$ is an eigenvector of $C_\psi$.
  If $\psi=\psi_s\circ\psi_u$ is the Jordan-Chevalley decomposition of $\psi$ with $\psi_s$ the semisimple part and $\psi_u$ the unipotent part of $\psi$, then $C_{\psi_s}$ is diagonalizable,  $C_{\psi_u}$ is unipotent and $C_\psi = C_{\psi_s}\circ C_{\psi_u}$ is the Jordan-Chevalley decomposition of $C_\psi$. Thus, $C_{\psi_s}(\rho(a)) =\gamma \rho(a)$ or
  $$
  \psi_s \circ \rho(a) = \rho(\sigma(a)) \circ \psi_s\,.
  $$
  Since $\sigma$ is diagonalizable, this equality holds for all $a \in A$. Therefore, $\psi_s$ is our desired isomorphism of $A$-modules.
\end{proof}

For the purpose of this paper and future development, we need to establish a more general version of a part of the results obtained in \cite[Lemmas 1.4 and 2.2(i)]{Ng08}.

\begin{lemma}\label{l:A2}
   Let $A$ be a finite-dimensional algebra over an  algebraically closed field $\KK$ of characteristic zero, and $\sigma$ an algebra automorphism on $A$ of order $n$. If $V$ is a finite-dimensional indecomposable $A$-module which is isomorphic to  $\ld \sigma V$, then there exists an isomorphism of $A$-modules $\phi: V \to \ld \sigma V$ such that $\phi^n=\id_V$.
 \end{lemma}

\begin{proof} We simply extract the proof from \cite[Lemma 1.4]{Ng08}. It follows by Lemma \ref{l:A1} that there exists a diagonalizable operator $\phi$ on $V$ such that $\phi(av)=\sigma(a)\phi(v)$ for $a \in A$ and $v\in V$. Thus, $\phi^n$ is an  $A$-module automorphism on $V$. Since $\End_A(V)$ is a finite-dimensional local $\KK$-algebra,  $\phi^n=c \cdot \id_V$ for some non-zero scalar $c$. Therefore, if we take $t \in \KK$ to be an $n$-th root of $c$, then $\frac{1}{t} \phi$ is a desired $A$-module isomorphism.
 \end{proof}
%\bibliography{myrefs}
%\bibliography{mybibl}

\begin{thebibliography}{Mas96b}

\bibitem[AN01]{AN01}
Nicol{\'a}s Andruskiewitsch and Sonia Natale, \emph{Counting arguments for
  {H}opf algebras of low dimension}, Tsukuba J. Math. \textbf{25} (2001),
  no.~1, 187--201. \MR{1846876 (2002d:16046)}

\bibitem[AS98]{AS98}
Nicol{\'a}s Andruskiewitsch and Hans-J{\"u}rgen Schneider, \emph{Hopf algebras
  of order {$p\sp 2$} and braided {H}opf algebras of order {$p$}}, J. Algebra
  \textbf{199} (1998), no.~2, 430--454. \MR{1489920 (99c:16033)}

\bibitem[BM89]{blattmontg}
Robert~J. Blattner and Susan Montgomery, \emph{Crossed products and {G}alois
  extensions of {H}opf algebras}, Pacific J. Math. \textbf{137} (1989), no.~1,
  37--54. \MR{983327 (90a:16007)}

\bibitem[EG98]{EG98}
Pavel Etingof and Shlomo Gelaki, \emph{Semisimple {H}opf algebras of dimension
  {$pq$} are trivial}, J. Algebra \textbf{210} (1998), no.~2, 664--669.
  \MR{1662308 (99k:16079)}
\bibitem[EG04]{EG03}
Pavel Etingof and Shlomo Gelaki, \emph{On {H}opf algebras of dimension {$pq$}},
  J. Algebra \textbf{277} (2004), no.~2, 668--674. \MR{2067625 (2005d:16061)}

\bibitem[F]{Fukuda}
Daijiro Fukuda, \emph{Classification of Hopf algebras of dimension 18}, to appear in Israel J. Math.

\bibitem[Gar05]{GG}
Gast{\'o}n~Andr{\'e}s Garc{\'{\i}}a, \emph{On {H}opf algebras of dimension
  {$p\sp 3$}}, Tsukuba J. Math. \textbf{29} (2005), no.~1, 259--284.
  \MR{2162840 (2006d:16056)}

\bibitem[GW00]{GW00}
Shlomo Gelaki and Sara Westreich, \emph{On semisimple {H}opf algebras of
  dimension {$pq$}}, Proc. Amer. Math. Soc. \textbf{128} (2000), no.~1, 39--47.
  \MR{2000c:16050}

\bibitem[KM97]{KM}
Toshiharu Kobayashi and Akira Masuoka, \emph{A result extended from groups to
  {H}opf algebras}, Tsukuba J. Math. \textbf{21} (1997), no.~1, 55--58.
  \MR{1467221 (98h:16064)}

\bibitem[LR87]{LaRa87}
Richard~G. Larson and David~E. Radford, \emph{Semisimple cosemisimple {H}opf
  algebras}, Amer. J. Math. \textbf{109} (1987), no.~1, 187--195.
  \MR{89a:16011}

\bibitem[LR88]{LaRa88}
Richard~G. Larson and David~E. Radford, \emph{Finite-dimensional cosemisimple {H}opf algebras in
  characteristic $0$ are semisimple}, J. Algebra \textbf{117} (1988), no.~2,
  267--289. \MR{89k:16016}

\bibitem[Mas95]{Mas95}
Akira Masuoka, \emph{Semisimple {H}opf algebras of dimension {$2p$}}, Comm.
  Algebra \textbf{23} (1995), no.~5, 1931--1940. \MR{1323710 (96e:16050)}


\bibitem[Mas96a]{Masuoka96}
Akira Masuoka, \emph{The {$p\sp n$} theorem for semisimple {H}opf algebras},
  Proc. Amer. Math. Soc. \textbf{124} (1996), no.~3, 735--737. \MR{1301036
  (96f:16046)}

\bibitem[Mas96b]{Masuoka2p^2}
Akira Masuoka, \emph{Some further classification results on semisimple {H}opf
  algebras}, Comm. Algebra \textbf{24} (1996), no.~1, 307--329. \MR{1370538
  (96k:16070)}

\bibitem[Mon93]{Montgomery}
Susan Montgomery, \emph{Hopf algebras and their actions on rings}, CBMS
  Regional Conference Series in Mathematics, vol.~82, Published for the
  Conference Board of the Mathematical Sciences, Washington, DC, 1993.
  \MR{1243637 (94i:16019)}

\bibitem[Mon01]{Mon01}
S.~Montgomery, \emph{Representation theory of semisimple {H}opf algebras},
  Algebra--representation theory ({C}onstanta, 2000), NATO Sci. Ser. II Math.
  Phys. Chem., vol.~28, Kluwer Acad. Publ., Dordrecht, 2001, pp.~189--218.
  \MR{1858037 (2002g:16062)}

\bibitem[Nat99]{Natalepq^2}
Sonia Natale, \emph{On semisimple {H}opf algebras of dimension {$pq\sp 2$}}, J.
  Algebra \textbf{221} (1999), no.~1, 242--278. \MR{1722912 (2000k:16050)}
\bibitem[Nat02]{Na4}
Sonia Natale, \emph{Quasitriangular {H}opf algebras of dimension {$pq$}}, Bull.
  London Math. Soc. \textbf{34} (2002), no.~3, 301--307. \MR{2003h:16066}

\bibitem[Ng98]{Ng98}
Siu-Hung Ng, \emph{On the projectivity of module coalgebras}, Proc. Amer. Math.
  Soc. \textbf{126} (1998), no.~11, 3191--3198. \MR{1469428 (99a:16035)}

\bibitem[Ng02]{Ng02}
Siu-Hung Ng, \emph{Non-semisimple {H}opf algebras of dimension {$p\sp 2$}}, J.
  Algebra \textbf{255} (2002), no.~1, 182--197. \MR{1935042 (2003h:16067)}

\bibitem[Ng04]{Ng04}
Siu-Hung Ng, \emph{Hopf algebras of dimension {$pq$}}, J. Algebra \textbf{276}
  (2004), no.~1, 399--406. \MR{2054403 (2005d:16066)}

\bibitem[Ng05]{Ng05}
Siu-Hung Ng, \emph{Hopf algebras of dimension {$2p$}}, Proc. Amer. Math. Soc.
  \textbf{133} (2005), no.~8, 2237--2242 (electronic). \MR{2138865
  (2006a:16055)}

\bibitem[Ng08]{Ng08}
Siu-Hung Ng, \emph{Hopf algebras of dimension {$pq$}. {II}}, J. Algebra
  \textbf{319} (2008), no.~7, 2772--2788. \MR{2397407}

\bibitem[NZ89]{NZ}
Warren~D. Nichols and M.~Bettina Zoeller, \emph{A {H}opf algebra freeness
  theorem}, Amer. J. Math. \textbf{111} (1989), no.~2, 381--385. \MR{987762
  (90c:16008)}
\bibitem[Rad75]{Rad75}
David~E. Radford, \emph{On a coradical of a finite-dimensional {H}opf algebra},
  Proc. Amer. Math. Soc. \textbf{53} (1975), no.~1, 9--15. \MR{0396652 (53
  \#514)}
\bibitem[Rad76]{Radford76}
David~E. Radford, \emph{The order of the antipode of a finite dimensional
  {H}opf algebra is finite}, Amer. J. Math. \textbf{98} (1976), no.~2,
  333--355. \MR{0407069 (53 \#10852)}

\bibitem[Rad85]{radfordproj}
David~E. Radford, \emph{The structure of {H}opf algebras with a projection}, J. Algebra
  \textbf{92} (1985), no.~2, 322--347. \MR{778452 (86k:16004)}

\bibitem[Rad94]{radfordtrace}
David~E. Radford, \emph{The trace function and {H}opf algebras}, J. Algebra \textbf{163}
  (1994), no.~3, 583--622. \MR{1265853 (95e:16039)}

\bibitem[Sch92]{Schn92}
Hans-J{\"u}rgen Schneider, \emph{Normal basis and transitivity of crossed
  products for {H}opf algebras}, J. Algebra \textbf{152} (1992), no.~2,
  289--312. \MR{1194305 (93j:16032)}

\bibitem[Sch93]{Sch93}
Hans-J{\"u}rgen Schneider, \emph{Some remarks on exact sequences of quantum groups}, Comm.
  Algebra \textbf{21} (1993), no.~9, 3337--3357. \MR{1228767 (94e:17026)}

\bibitem[Swe69]{Sweedler}
Moss~E. Sweedler, \emph{Hopf algebras}, Mathematics Lecture Note Series, W. A.
  Benjamin, Inc., New York, 1969. \MR{0252485 (40 \#5705)}
\bibitem[Taf71]{Taft71}
Earl~J. Taft, \emph{The order of the antipode of finite-dimensional {H}opf
  algebra}, Proc. Nat. Acad. Sci. U.S.A. \textbf{68} (1971), 2631--2633. \MR{44
  \#4075}

\bibitem[Zhu94]{Zhu94}
Yongchang Zhu, \emph{Hopf algebras of prime dimension}, Internat. Math. Res.
  Notices (1994), no.~1, 53--59. \MR{1255253 (94j:16072)}

\end{thebibliography}
%\bibliographystyle{amsalpha}
%\end{document}
\providecommand{\bysame}{\leavevmode\hbox to3em{\hrulefill}\thinspace}
\providecommand{\MR}{\relax\ifhmode\unskip\space\fi MR }
\providecommand{\MRhref}[2]{%
  \href{http://www.ams.org/mathscinet-getitem?mr=#1}{#2}
}
\providecommand{\href}[2]{#2}

\end{document}